\documentclass[A4,12pt]{article}
\usepackage{amsmath,amssymb,amsfonts,amsthm,graphicx}
\usepackage{color}
\usepackage{epsfig}
\usepackage{epstopdf}

\usepackage[top=2cm, bottom=2cm, left=2cm, right=2cm]{geometry}

\newtheorem{theorem}{Theorem}[section]
\newtheorem{lemma}[theorem]{Lemma}

\theoremstyle{definition}
\newtheorem{definition}[theorem]{Definition}

\newtheorem{corollary}[theorem]{Corollary}
\newtheorem{example}[theorem]{Example}

\theoremstyle{remark}

\usepackage{algorithm}
\usepackage{algpseudocode}
\algnewcommand\algorithmicswitch{\textbf{switch}}
\algnewcommand\algorithmiccase{\textbf{case}}
\algnewcommand\algorithmicassert{\texttt{assert}}
\algnewcommand\Assert[1]{\State \algorithmicassert(#1)}%

\algdef{SE}[SWITCH]{Switch}{EndSwitch}[1]{\algorithmicswitch\ #1\ \algorithmicdo}{\algorithmicend\ \algorithmicswitch}%
\algdef{SE}[CASE]{Case}{EndCase}[1]{\algorithmiccase\ #1}{\algorithmicend\ \algorithmiccase}%
\algtext*{EndSwitch}%
\algtext*{EndCase}%

\title{A note on bipartite graphs whose [1, k]-domination number equal to their number of vertices}

\author{Narges Ghareghani$^{*,1}${\footnote{{$^*$Corresponding author.}}},  Iztok Peterin$^{3,4}$ and Pouyeh Sharifani$^{5}$\\
\footnotesize{$^{1}$ Department of
Industrial design, College of Fine Arts, University of Tehran, Tehran, Iran} \\
\footnotesize{$^{3}$ Faculty of Electrical Engineering and Computer Science, University of Maribor, Maribor, Slovenia}\\
\footnotesize{$^{4}$ Institute of Mathematics, Physics and Mechanics, Ljubljana, Slovenia}\\
\footnotesize{$^{5}$Department of Computer Science, Yazd University, Yazd, Iran.}\\
\footnotesize{E-mails:, ghareghani@ut.ac.ir, ghareghani@ipm.ir, iztok.peterin@um.si, pouyeh.sharifani@gmail.com}}
\date{}

\begin{document}

	\maketitle
	
\begin{abstract}
A subset $D$ of the vertex set $V$ of a graph $G$ is called an $[1,k]$-dominating set if every vertex from $V-D$ is adjacent to at least one vertex and at most $k$ vertices of $D$. A $[1,k]$-dominating set with the minimum number of vertices is called a $\gamma_{[1,k]}$-set and the number of its vertices is the $[1,k]$-domination number $\gamma_{[1,k]}(G)$ of $G$.  In this short note we show that the decision problem whether $\gamma_{[1,k]}(G)=n$ is an $NP$-hard problem, even for bipartite graphs. Also, a simple construction of a bipartite graph $G$ of order $n$ satisfying $\gamma_{[1,k]}(G)=n$ is given for every integer $n\geq (k+1)(2k+3)$.
\end{abstract}



\section{Introduction}

Let $G$ be a graph with vertex set $V(G)$ and edge set $E(G)$. A subset $D$ of $V(G)$ is called a \emph{dominating set}, if every vertex from $V(G)-D$ has at least one neighbor in $D$. The minimum cardinality of a dominating set is called the \emph{domination number} of $G$ and is denoted by $\gamma(G)$. A dominating set $D$ of $G$ is called a $[1,k]$-\emph{dominating set} if every vertex of $V-D$ is adjacent to at most $k$ vertices of $D$. The minimum cardinality of a $[1,k]$-dominating set is  the $[1,k]$-\emph{domination number} of $G$ and  denoted by $\gamma_{[1,k]}(G)$. We call a $[1,k]$-dominating set of cardinality $\gamma_{[1,k]}(G)$ a $\gamma_{[1,k]}(G)$-\emph{set}. Clearly $\gamma(G) \leq \gamma_{[1,k]}(G) \leq |V(G)|$, which are the trivial bounds for $\gamma_{[1,k]}(G)$.

The invariant $\gamma_{[1,k]}(G)$ was introduced by Chellali et al. in \cite{chellali20131} in the more general setting of the $[j, k]$-domination number of a graph. They proved that computing $\gamma_{[1,2]}(G)$ is an NP-complete problem. Among other results, it was shown that the trivial bounds are strict for some graphs in the case of $k=2$. They also posed several  questions; one of them was to characterize graphs for which the trivial lower bound is strict for $k=2$, that is $\gamma_{[1,2]}(G)=\gamma(G)$. Recently, see \cite{etesami2019optimal}, it was shown that there is no polynomial recognition algorithm for graphs graphs with $\gamma_{[1,k]}(G)=\gamma(G)$ unless $P=NP$.

Some other  problems from \cite{chellali20131} have been considered in \cite{bishnu, chellali2014independent, goharshady, yang20141}. For instance, in \cite{yang20141} authors find planar graphs and bipartite graphs of order $n$ with $\gamma_{[1,2]}(G)=n$. More precisely, for integer $n$ which is sufficiently large, they construct a bipartite graph $G$ of order $n$ with $\gamma_{[1,2]}(G) =n$. The construction is complicated and work only for large $n$. 

In this note we present a simple construction of a bipartite graph $G$ of order $n$ with $\gamma_{[1,k]}(G)=n$ for any integers $k\geq 2$ and  $n\geq (k+1)(2k+3)$. Hence, we generalize and simplify some results given in \cite{yang20141}. We also show that the decision problem $\gamma_{[1,k]}(G)=n$ is NP-hard for a given bipartite graph $G$ of order $n$, $n>k\geq 2$.

\section{Preliminaries}\label{notation}

Let $G$ be a simple graph with vertex set $V(G)$ and edge set $E(G)$. An {\it empty graph} on $n$ vertices $\overline{K_n}$ consists of $n$ isolated vertices with no edges. A tree which has exactly one vertex of degree greater than two is said to be  {\it star-like}. The vertex of maximum degree of such a tree is called the {\it central vertex}. The graph $T-v$, where $T$ is a star-like tree and $v$ its central vertex, contains disjoint paths $P_{n_1}, \ldots ,P_{n_k}$ and is denoted by $S(n_1,\ldots ,n_k)$. 

A subset $D$ of $V(G)$ is called a {\it total dominating set} if every $v\in V(G)$ is adjacent to a vertex from $D$. The minimum cardinality of a total dominating set in graph $G$ is denoted by $\gamma_t(G)$ and is called the \emph{total domination number}. A total dominating set $D\subseteq V(G)$ is a \emph{total} $[1,k]$-\emph{dominating set}, if for every vertex $v\in V(G)$ is adjacent to at most $k$ vertices from $D$.  While an $[1,k]$-dominating set exists for every graph $G$, there exist graphs which do not have any total $[1,k]$-dominating sets. By $\gamma_{t[1,k]}(G)$ we denote the minimum cardinality of a total $[1,k]$-dominating set (if it exist), it is $\infty$ if no total $[1,k]$-dominating set exists. An example of a graph $G_1\cong S(2,2,2)$ with $\gamma_{t[1,2]}(G_1)=\infty$ is presented on Figure \ref{figS(2,2,2)}.

\begin{figure}[h!]
   \centering
     \includegraphics[width=0.3\textwidth]{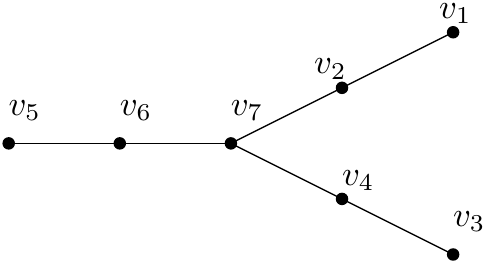}
      \caption{Bipartite graph $G_1\cong S(2,2,2)$ with $\gamma_{t[1,2]}(G_1)=\infty$.}
      \label{figS(2,2,2)}
\end{figure}



The {\it lexicographic product} of two graphs $G$ and $H$, denoted by $G\circ H$, is a graph with the vertex set  $V(G \circ H) = V(G) \times V(H)$, where two vertices $(g,h)$ and $(g',h')$ are adjacent in $G\circ H$ if $gg'\in E(G)$ or $g=g'$ and $hh'\in E(H)$. It follows directly from the definition of the lexicographic product that $G\circ H$ is bipartite if and only if one factor is the empty graph $\overline{K_t}$ and the other is bipartite. Moreover, for a graph $G$ on at least two vertices, the graph $G\circ H$ is connected and bipartite if and only if $G$ is connected and bipartite and $H\cong\overline{K_t}$. See \cite{hammack2011handbook} for more information about lexicographic and other products.

For any $h_0 \in V(H)$, we call the set $G^{h_0}=\{(g,h_0)\in V(G \circ H): g\in V(G)\}$ a {\it$G$-layer} of the graph $G\circ H$. Similarly, for $g_0 \in V(G)$, we call the set $H^{g_0} = \left\{ (g_0,h) \in V(G \circ H) : h \in V(H) \right\}$ an {\it $H$-layer} of the graph $G\circ H$.

Recently, see \cite{GIPS}, $\gamma_{[1,k]}(G\circ H)$ was described as an optimization problem of some partitions of $V(G)$. For some special cases it is possible to present $\gamma_{[1,k]}(G\circ H)$ as an invariant of $G$. In particular this is possible when $\gamma_{[1,k]}(H)>k$ and $H$ contains an isolated vertex.

\begin{theorem}\label{case1}(\cite{GIPS}, Theorem 4.4)
Let $G$ be a connected graph, $H$ a graph and $k\geq 2$ an integer. If $\gamma_{[1,k]}(H)>k$, and $H$ contains an isolated vertex, then
\begin{equation*}
\gamma_{[1,k]}(G\circ H)=\left\{
\begin{array}{ccc}
\gamma_{t[1,k]}(G) & : & \gamma _{t[1,k]}(G)<\infty  \\
|V(G)|\cdot|V(H)| & : & \text{otherwise}%
\end{array}%
\right. .
\end{equation*}
\end{theorem}

The following corollary is the direct consequence of Theorem \ref{case1} and will be useful later to construct bipartite graphs with $\gamma_{[1,k]}(G)=\vert V(G)\vert$.

\begin{corollary}\label{reduction-proof}
Let $G$ be a connected graph and $H\cong\overline{K_{k+1}}$. Then $\gamma_{[1,k]}(G \circ H)=\vert V(G \circ H)\vert$ if and only if $G$ has no total $[1,k]$-dominating set.
\end{corollary}

\section{Complexity}\label{complexity}

In this section we will show that it is $NP$-hard to check whether $\gamma_{[1,k]}(G)=\vert V(G) \vert$ for a bipartite graph $G$.
For this aim, we first show that a related problem to check whether $\gamma_{t[1,k]}(G)=\vert V(G) \vert$ is $NP$-hard for a bipartite graph $G$. The problem is called a BipTotal $[1,k]$-set problem. To prove this we use reduction from a kind of set cover problem, called $[1,k]$-triple set cover problem, which is known to be  $NP$-hard as shown in \cite{golovach2007computational}. Then, using Theorem \ref{case1}, we prove that for a bipartite graph $G$, checking whether $\gamma_{[1,k]}(G)=\vert V(G) \vert$ is an $NP$-hard problem.

\begin{alignat*}{3}
&\text {\it Problem A:\,\,\,\,\,\,\,} &&[1,k]\text{-triple set cover}\\
&\text {\it Input:} &&\text {A finite set } X=\{x_1,\ldots, x_{n} \}  \text { and a collection } C=\{C_1,\ldots, C_t\} \\
& &&\text {of 3-element subsets of } X.\\
&\text {\it Output:} &&\text {{\bf Yes} if there exists a }  C'\subseteq C \text{ such that every element of } X
\text{ appears in }
 \\
 &   && \text{at least one and at most } k\text{ elements of } C',\text{{\bf No} otherwise.}
\end{alignat*}

\begin{alignat*}{3}
&\text {\it Problem B:\,\,\,\,\,\,\,} &&\text{BipTotal }[1,k]\text{-set}\\
&\text {\it Input:} &&\text {A bipartite graph } G. \\
&\text {\it Output:} &&\text {{\bf Yes}  if there exists a } D\subseteq V(G) \text{ such that every element of } V(G) \text{ is adjacent to } \\
	&  &&\text{ at least one and at most } k\text{ vertices of } D,\text{{\bf No} otherwise.}
\end{alignat*}

We are going to prove that the BipTotal $[1,k]$-set  problem is $NP$-hard by giving a polynomial time reduction from $[1,k]$-triple set cover problem.

\begin{definition}
Let $X = \{x_{1},\ldots, x_{n}\}$ and $C = \{C_{1}, \ldots,C_{t}\}$ be any given instance of Problem A. We construct a graph $G_{X,C}$ as follows:
$$ V(G_{X,C})=\bigcup_{i=1}^{t} (P_i \cup L_i) \cup X \cup \{c_1, \ldots, c_t\},$$
where for each integer $i$, $1\leq i\leq t$, we have $P_i=\{p_{i,1}, \ldots, p_{i,k}\}$, $L_i=\{l_{i,1}, \ldots, l_{i,k}\}$, and
	$$E(G_{X,C})=\bigcup_{1 \leq j \leq t}\{c_jp_{j,1}, \ldots , c_jp_{j,k},p_{j,1}l_{j,1}, \ldots , p_{j,k}l_{j,k} \} \cup \bigcup_{1 \leq i,j \leq t } \{x_{i}c_{j}: x_{i}\in C_{j}\}.$$
\end{definition}

\begin{lemma}\label{setcovertotal}
	 Let $X = \{x_{1},\ldots,x_{n}\}$ be a finite set and $C = \{C_{1}, \ldots, C_{t}\}$ be a collection of $3$-element subsets of $X$. Problem A for $(X,C)$ is a YES instance if and only if $G_{X,C}$ is a YES instance of Problem B.
\end{lemma}

\begin{proof}
{Suppose that $C'$ is a solution for the instance $(X,C)$ of Problem A. We construct $D$ as follows:
	$$D= \bigcup_{1 \leq j \leq t}P_j \cup \bigcup_{C_{j} \in C'}\{c_j\} \cup \bigcup_{C_{j} \notin C'}L_j.$$
	
	We can check easily  that $D$ is a $[1, k]$-total set for $G_{X,C}$.
	Conversely, suppose that $G_{X,C}$ has a total $[1, k]$-set $D$. Clearly $D$ must contain all vertices of $P_j$ because every $p_{j,j'}$ is adjacent to at least one leaf $l_{j,j'}$.
	These vertices dominate every $c_j$ exactly $k$ times. Therefore, there is no vertex $x_i$ in $D$; in other words $D\cap X= \emptyset$. So, every $x_i$ must be dominated by a vertex of $\{c_1, \cdots , c_t\}$. It is easy to see that there is a solution  $C' \subseteq C$ for $[1,k]$-triple set cover problem if and only if  the corresponding vertices $C'$ of $V(G)$  dominate all vertices of $\{x_1, \ldots , x_n\}$ at least once and at most $k$ times.
These vertices dominate all vertice $\{p_{j,1}, \ldots, p_{j,k}\}$ for $c_j \in C'$. To dominate all other vertices we add $\{l_{j,1}, \ldots, l_{j,k}\}$  to $D$ for $c_j \notin C'$.}
\end{proof}

The following example help us to understand the definition and the Lemma.
\begin{example}\label{example}
Let $X=\{x_1,x_2,x_3,x_4,x_5,x_6,x_7,x_8,x_9\}$ and $C=\{C_1,C_2,C_3,C_4,C_5\}$ such that $C_1=\{x_1,x_2,x_4\}$, $C_2=\{x_2,x_5,x_7\}$, $C_3=\{x_4,x_5,x_6\}$, $C_4=\{x_3,x_5,x_9\}$ and $C_5=\{x_3,x_8,x_9\}$. For $k=3$, the corresponding graph $G_{X,C}$ is shown in Figure \ref{fig1}.
 This  is a YES-instance for the $[1,3]$-set cover problem $(X,C)$, because $C'=\{C_1,C_2,C_3,C_5\}$ has the desired property. The vertices of total $[1,3]$-set are black vertices shown in the figure.

\begin{figure}[h!]
	\centering
	\includegraphics[width=0.95\textwidth]{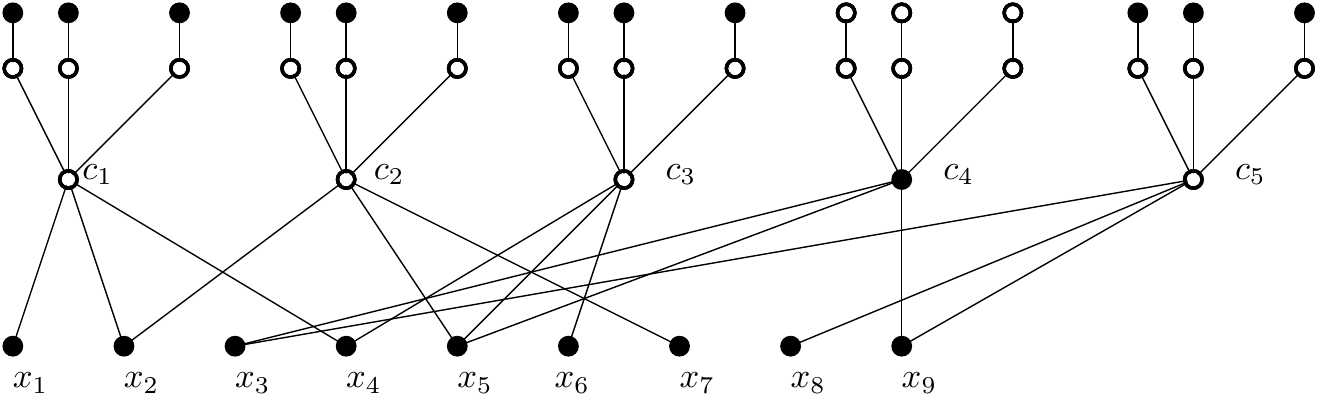}
	\caption{$G_{X,C}$ from Example \ref{example}.}
	
	\label{fig1}
\end{figure}

\end{example}

\begin{theorem}\label{Complexity2}
	The BipTotal $[1, k]$-set  problem is $NP$-hard.
\end{theorem}	
\begin{proof}

 By Lemma 1 of \cite{golovach2007computational} the $[1,k]$-triple set cover problem is $NP$-hard. Hence, using Lemma \ref{setcovertotal} the BipTotal $[1, k]$-set  problem is also $NP$-hard.
\end{proof}


The following theorem which is the main result of this section is a direct consequence of  Theorem~\ref{Complexity2} and Corollary \ref{reduction-proof}.
\begin{theorem}\label{Complexityequal1}
 For bipartite graphs it is $NP$-hard to decide whether $\gamma_{[1,k]}(G)=\vert V(G) \vert$.
\end{theorem}

\section{Construction}

Here, for any integers $k\geq 2$ and $n\geq (k+1)(2k+3)$,  we construct a bipartite graph $G$ of order $n$ with $\gamma_{[1,k]}(G)=n$. As already mentioned, in \cite{yang20141} a bipartite graph $G$ of order $n$ was constructed for sufficiently large integer $n$ which satisfies $\gamma_{[1,2]}(G)=n$.

First, we give our construction in the case of $k=2$, then we extend the result to the general case.

\begin{example}\label{exam21}
If $G_1\cong S(2,2,2)$, see Figure \ref{figS(2,2,2)}, and $H\cong\overline{K_3}$, then $G=G_1\circ H$, see Figure \ref{figprod21}, is bipartite and $\gamma_{[1,2]}(G)=|V(G)|$ by Theorem \ref{case1}.

\begin{figure}[h!]
   \centering
     \includegraphics[width=0.4\textwidth]{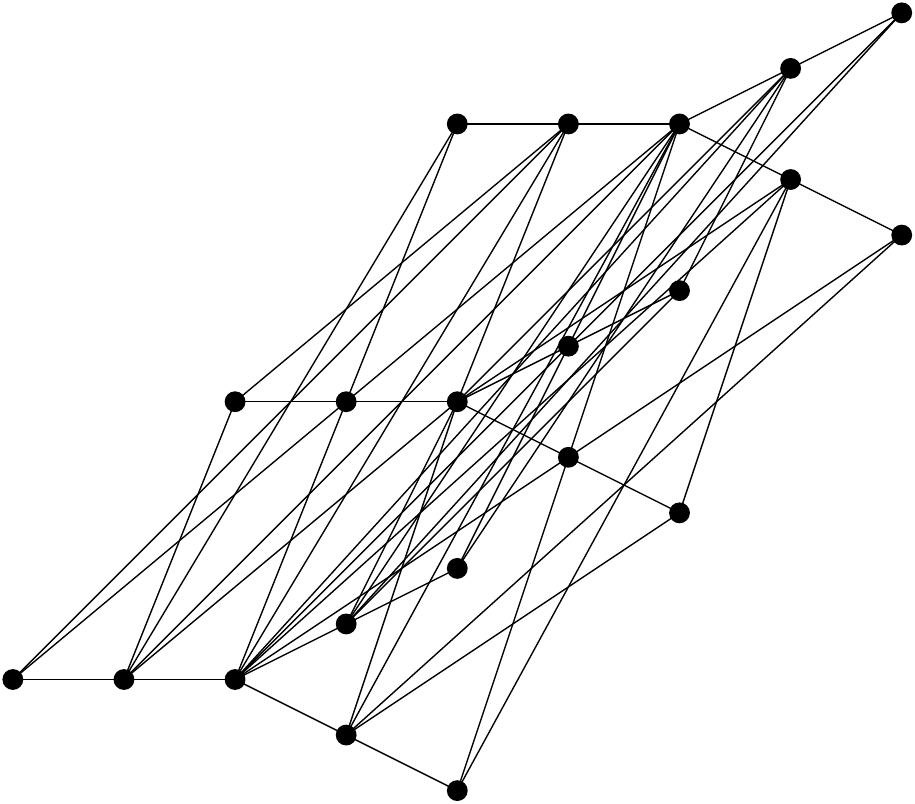}
      \caption{Bipartite graph $G$ with $\gamma_{[1,2]}(G)=|V(G)|$.}
      \label{figprod21}
\end{figure}

\end{example}

\begin{theorem}\label{n>21}
For any integer $n\geq 21$, there exits a bipartite graph $\Gamma$ with $n$ vertices such that $\gamma_{[1,2]}(\Gamma)=n$.
\end{theorem}

\begin{proof}
Let $G_1$ and $G$ be graphs shown on Figures \ref{figS(2,2,2)} and \ref{figprod21}, respectively. By Example \ref{exam21} $G$ is a bipartite graph with $21$ vertices for which $\gamma_{[1,2]}(G)=21$. Let $v_1,\ldots, v_7\in V(G_1)$ be the vertices of $G_1$ as shown on Figure \ref{figS(2,2,2)}. For any integer
$t\geq 1$, using the graph $G$ of Figure \ref{figprod21}, we construct a new bipartite graph $\Gamma$ of order $n=21+t$
as follows: $\Gamma=(V(\Gamma), E(\Gamma))$, where $V(\Gamma)=V(G) \cup \{a_1, \ldots, a_t\}$ and $E(\Gamma)=  \bigcup_{h\in H}\{a_1(v_2,h), \ldots, a_t(v_2,h)\}\cup E(G) $ (see Figure \ref{figprod21+t}).
\begin{figure}[h!]
   \centering
     \includegraphics[width=0.6\textwidth]{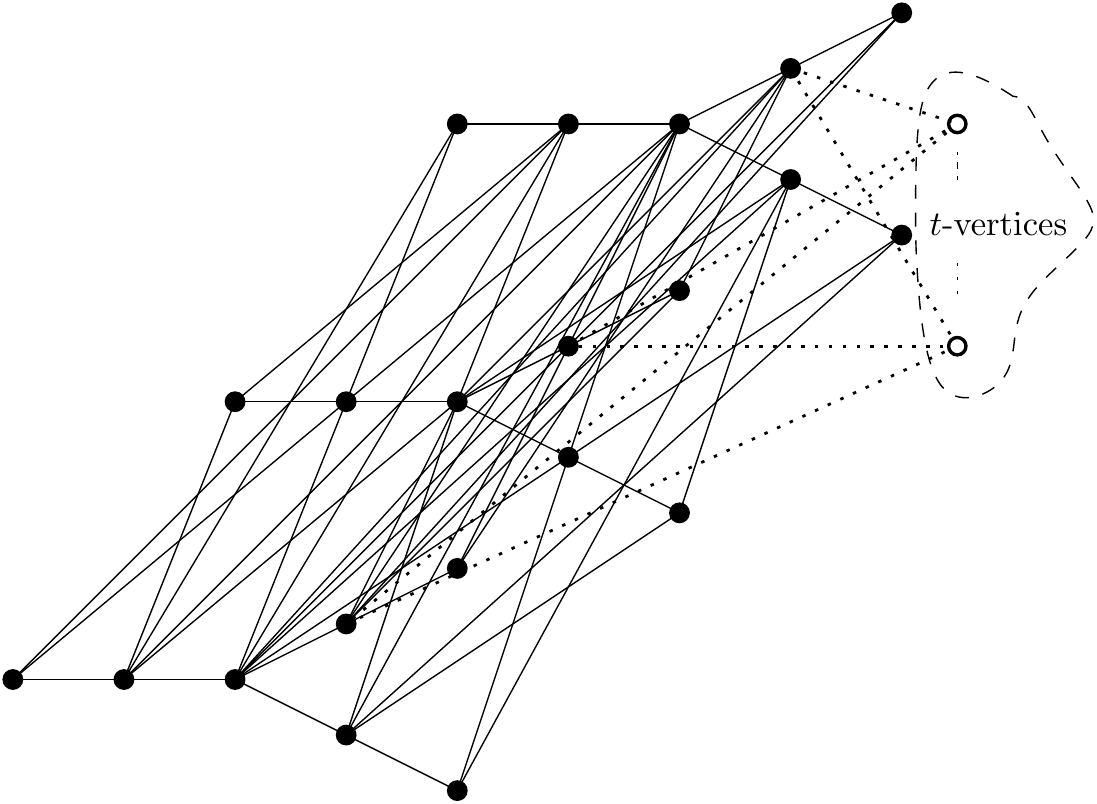}
      \caption{Bipartite graph $\Gamma$ with $\gamma_{[1,2]}(\Gamma)=|V(\Gamma)|$.}
      \label{figprod21+t}
\end{figure}

Let $S$ be a $[1,2]$-set for $\Gamma$. First, we claim that there exist a vertex $h\in H$ with $(v_2, h)\in S$.
To dominate the three vertices of the $H$-layer $H^{v_1}$, either there exits vertex $h\in H$ with $(v_2, h)\in S$ or $H^{v_1}\subseteq S$. If there exits vertex $h\in H$  with $(v_2, h)\in S$, then there is nothing to prove.
If $H^{v_1}\subseteq S$, then every vertex of $H^{v_2}$ is dominated at least three times, hence $H^{v_2}\subseteq S$. Therefore the claim is true and
$H^{v_2}\cap S \neq \emptyset$.
By the same reasoning we have
$H^{v_4}\cap S \neq \emptyset$ and $H^{v_6}\cap S \neq \emptyset$. Hence, by the definition of lexicographic product of graphs, every vertex of $H^{v_7}$ is dominated at least three times. Therefore, we have
\begin{equation}\label{v7}
    H^{v_7}\subseteq S.
\end{equation}

Now, by (\ref{v7}) every vertex in $H^{v_2}\cup H^{v_4} \cup H^{v_6}$ is dominated at least three times and so we have
\begin{equation}\label{v246}
H^{v_2}\cup H^{v_4} \cup H^{v_6}\subseteq S.
\end{equation}
And, then by (\ref{v246}) we conclude that:
\begin{equation}\label{v135}
H^{v_1}\cup H^{v_3} \cup H^{v_5} \cup  \{a_1, \ldots, a_t\} \subseteq S.
\end{equation}

Therefore, $S=V(\Gamma)$, as desired.
\end{proof}

We end with a generalization of the above result from $k=2$ to $k\geq 2$.

\begin{theorem}
For integers $k\geq 2$ and $n\geq (k+1)(2k+3)$ there exists a bipartite graph $\Gamma$ with $n$ vertices such that $\gamma_{[1,k]}(\Gamma)=n$.
\end{theorem}

\begin{proof}
Let $G_1=S(2,2, \ldots, 2)$ be a star-like tree with $2k+3$ vertices and let $H\cong\overline{K_{k+1}}$. Clearly $G=G_1\circ H$ is a bipartite graph. Let $v_1\in V(G_1)$ be a vertex of degree one and  $v_2\in V(G_1)$ be its only neighbor. For any integer $t\geq 1$, using the graph $G$, we  construct a new bipartite graph $\Gamma$ of order $n=(k+1)(2k+3)+t$ as follows: $\Gamma=(V(\Gamma), E(\Gamma))$, where $V(\Gamma)=V(G) \cup \{a_1, \ldots, a_t\}$ and $E(\Gamma)=  \bigcup_{h\in H}\{a_1(v_2,h), \ldots, a_t(v_2,h)\}\cup E(G) $.

Let $S$ be a $\gamma_[1,k](\Gamma)$-set. By the same reasoning as in the proof of Theorem \ref{n>21} one can show that $H^{v_i}\cap S \neq \emptyset$ for every vertex $v_i\in V(G_1)$ with $\deg_{G_1}(v_i)=2$. Since there are $k+1$ such vertices in $G_1$, all vertices of $H^v$ must be in $S$ for a central vertex $v$ of $G_1$. This clearly leads to $\gamma_{[1,k]}(\Gamma)=|V(\Gamma)|$ because $|H^v|=k+1$.
\end{proof}


	\end{document}